\documentclass[11pt,a4paper]{amsart}
\newcommand{\correctspacing}{\singlespacing}

\usepackage{amsmath,amsthm,amsfonts,amssymb,epsfig}
\usepackage{setspace}
\usepackage{xcolor}

\usepackage{graphicx} 
\usepackage{amssymb}
\usepackage{graphicx}
\usepackage{color}

\usepackage{amssymb}
\usepackage{amsmath}
\usepackage{amsthm}
\usepackage{amscd}
\usepackage{url}
\usepackage{enumerate}

\theoremstyle{plain}
\newtheorem{theorem}{Theorem}[section]

\newtheorem{lemma}[theorem]{Lemma}
\newtheorem{que}[theorem]{Question}

\theoremstyle{definition}
\newtheorem{definition}[theorem]{Definition}

\theoremstyle{definition}

\correctspacing

\providecommand{\RR}{\mathbb{R}}

\providecommand{\cG}{\mathcal{G}}

\providecommand{\cM}{\mathcal{M}}

\newcommand{\be}{\begin{equation}} \newcommand{\beo}{\begin{equation*}}
\newcommand{\ee}{\end{equation}}    \newcommand{\eeo}{\end{equation*}}
\newcommand{\bea}{\begin{eqnarray}}
\newcommand{\eea}{\end{eqnarray}}
\newcommand{\beao}{\begin{eqnarray*}} \newcommand{\eeao}{\end{eqnarray*}}

\author{}
\author{Claus Griessler}
\thanks{Financial support through FWF-projects  Y782 and P26736 is  acknowledged.}
\date{\today}
\title[Monotonicity as a sufficiency criterion]{$C$-cyclical monotonicity as a sufficient criterion for optimality in the multi-marginal Monge-Kantorovich problem}

\begin{document}

\maketitle

\begin{abstract} This note establishes  that  a generalization of $c$-cyclical monotonicity from the Monge-Kantorovich problem with two marginals gives rise to a sufficient condition for optimality also in the multi-marginal version of that problem. To obtain the result, the cost function is assumed to be bounded by a sum of integrable functions. The proof rests on ideas from martingale transport.
\end{abstract}

\section{Introduction and result}

Let $X_1, \dots, X_d$ be Polish spaces, and $\mu_1, \dots, \mu_d$ probability measures on their Borel-$\sigma$-fields. By $\cM (\mu_1, \dots, \mu_d)$ we denote the set of probability measures on the space $E = X_1 \times \dots \times X_d$ with marginal distributions $\mu_1, \dots, \mu_d$. Writing $p_i$ for the canonical projections $E \rightarrow X_i$,  a measure $\mu$ on $E$ is in $\cM (\mu_1, \dots, \mu_d)$ if and only if 

\beo
p_i (\mu) = \mu_i ~ \text{ for  } i=1, \dots, d.
\eeo
These measures are called transports or transport plans. Given a measurable cost function $c: E \rightarrow \RR$, the cost of a transport $\mu$ is the integral $\int c \; d\mu$. The multi-marginal Monge-Kantorovich problem is to minimize the cost amongst transports, i.e. to solve

\be{\label{MK}}{\tag{mmMK}}
\inf_{\mu \in \cM (\mu_1, \dots, \mu_d)} \int c \, d\mu.
\ee
There is  a huge literature for the case $d=2$, \emph{the} Monge-Kantorovich problem, see e.g. \cite{Vi03}, \cite{Vi09} or \cite{AmGi13} for an overview. The literature on the case $d > 2$ is more recent and less voluminous. For an overview the reader is referred to \cite{Pa14}. \\
For $d=2$,  a characterization of optimal transport plans is given by the concept of $c$-cyclical monotonicity, see \cite[Ch.\ 5]{Vi09}: under fairly weak assumptions on the cost function, a transport is optimal if and only if it is $c$-cyclically monotone. A transport  is $c$-cyclically monotone if it is concentrated on a $c$-cyclically monotone set $\Gamma \subseteq X_1 \times X_2 = X \times Y$, i.e. a set $\Gamma$ such that  for any pairs $(x_1, y_1), \dots, (x_n, y_n) \in \Gamma$ one has, with $y_{n+1} = y_1$, 

\be
\sum_{i=1}^n c(x_i, y_i) \leq \sum_{i=1}^n c(x_i, y_{i+1}).
\ee

Does such  a characterization also hold for the case $d > 2$? \\
We start with a definition with a built-in mini-theorem that is well-known for $d=2$ and similarly easy to show for $d>2$:\footnote{In order to show (ii) from (i) it is enough to deal with measures $\alpha$ and $\alpha'$ that assume rational values only. One multiplies both $\int c \, d\alpha$ and $\int c \, d \alpha'$ with the integer $\tau$ which is defined as the product of all the denominators appearing in the values of $\alpha$ and $\alpha'$. It is then possible to write $\tau \int c \, d \alpha$ as a sum of the form $\sum_{i=1}^n c(x_1^{(i)}, \dots, x_d^{(i)})$ and because of the assumptions on $\alpha$ and $\alpha'$ one can find permutations to write $\tau \int c \, d\alpha'$ as $\sum_{i=1}^n c ( x_1^{(i)}, x_2^{(\sigma_2(i))}, \dots, x_d^{(\sigma_d(i))})$.}

\begin{definition}
A set $\Gamma \subseteq E$ is $c$-cyclically monotone if it fulfills any of the two following equivalent conditions:
\begin{enumerate}[(i)]
\item 
 for any $n$ and any points $\bigl(x_1^{(1)}, \dots, x_d^{(1)}\bigr), \dots, \bigl(x_1^{(n)}, \dots, x_d^{(n)}\bigr)\in \Gamma$ and  permutations $\sigma_2, \dots, \sigma_{d}: \{1, \dots, n \} \rightarrow \{1, \dots, n\}$, one has
\beo
\sum_{i=1}^n c\bigl(x_1^{(i)}, \dots, x_d^{(i)}\bigr) \leq \sum_{i=1}^n c\bigl( x_1^{(i)}, x_2^{(\sigma_2 (i))}, \dots, x_d^{(\sigma_d (i))} \bigr).
\eeo
\item 
any finite measure $\alpha$  concentrated on finitely many points in $\Gamma$ is a cost-minimizing transport plan between its marginals; i.e. if $\alpha'$ has the same marginals as $\alpha$, then 
\beo
\int c \, d\alpha \leq \int c \, d\alpha'.
\eeo
\end{enumerate}
\end{definition}

A weaker notion of $c$-monotonicity allowing only comparisons of two points in (i) was shown to be a necessary condition for optimality in \cite{Pa12fm}, see also \cite{CoPaMa15}. The necessity of $c$-cyclical  monotonicity in the sense of (i) is included in the results in \cite{BeGr14, Za14}. Cyclical monotonicity was also discussed in \cite{KiPa13} where cost functions that satisfy the twist condition on cyclically monotone or on splitting sets are shown to have a unique Monge solution of \eqref{MK}, but the exact connection between splitting sets and cyclically monotone sets  remains an open question. It is answered here as a byproduct in Lemma 2.5.\\
The question of sufficiency of cyclical monotonicity of  a transport plan for optimality  was open, although there was an early result in \cite{KnSm94} for quadratic costs in the case $d=3$. The situation is hence somewhat similar to the two-marginals case, where the sufficiency of $c$-cyclical monotonicity was open for some time and is now known to require more regularity of the cost function, see \cite{AmPr03, Pr07, ScTe08, BGMS09, BiCa10, Be15}.

In order to prove the sufficiency of $c$-monotonicity for optimality, we assume $c$ to be continuous and  \emph{bounded by a sum of integrable functions}. This means that there are functions $f_i \in L_1(\mu_i)$ such that 
\beo
c (x_1, \dots, x_d) \leq f_1 (x_1) + \dots + f_d (x_d) ~\text{ for all } ~x_1, \dots, x_d.
\eeo

\begin{theorem}
Let $c$ be a continuous cost-function $E \rightarrow [0, \infty)$ which is bounded  by a sum of integrable functions. Let  $\mu$ be a $c$-monotone transport plan in $\cM (\mu_1, \dots, \mu_d)$. Then $\mu$ is optimal. 
\end{theorem}

\section{Proof of Theorem 1.2}

The proof of Theorem 1.2 takes the proof for the case $d=2$ in \cite{ScTe08} as a blueprint: we show that $c$-cyclically monotone sets are \emph{$c$-splitting} sets. Optimality then follows easily from the assumptions on $c$. We  exploit ideas found in \cite{BeJu14}, where a notion of finite optimality is introduced as a generalization of $c$-cyclical monotonicity to the martingale-transport problem (with two marginals). 
The compactness-argument to show that $c$-cyclically monotone sets are $c$-splitting  is an adapted version of the argument in \cite{BeJu14} to show that finitely optimal sets are ``$c$-good". It is maybe worthy the mentioning that, although the arguments from \cite{BeJu14} can be adapted to the multi-marginal Monge-Kantorovich problem, it is an open question whether this is also possible for the multi-marginal martingale problem. 

\begin{definition}
A set $G \subseteq E$ is called $c$-splitting if there exist $d$  functions $\varphi_{i}: X_{i} \rightarrow [-\infty, \infty)$ such that 
\beo
\varphi_{1} (x_{1}) + \varphi_{2}(x_{2}) + \dots + \varphi_{d}(x_{d}) \leq c(x_{1}, x_{2}, \dots, x_{d}) 
\eeo
holds for all $(x_{1}, x_{2}, \dots, x_{d}) \in E$, and
\beo
\varphi_{1} (x_{1}) + \varphi_{2}(x_{2}) + \dots + \varphi_{d}(x_{d}) = c(x_{1}, x_{2}, \dots, x_{d}) 
\eeo holds for all $(x_{1}, \dots, x_{d}) \in G$.
We call the functions $(\varphi_{1}, \dots, \varphi_{d})$ a $(G,c)$-splitting tuple.\\ 
\end{definition}
The definition of splitting tuples comes without regularity assumptions on the functions $\varphi_i$. If the functions in a $(G,c)$-splitting tuple are measurable, we call it a measurable tuple. The next lemma shows that for continuous $c$ measurability comes at no cost:
\begin{lemma}
If $G$ is a $c$-splitting set and $c$ is continuous, then there is a measurable $(G,c)$-splitting tuple.
\end{lemma}
\begin{proof}
There is a $c$-splitting tuple $(\varphi_1, \dots, \varphi_d)$ by assumption. 
Set 
$$
\tilde{\varphi}_1 (x_1^0) = \inf_{x_2, \dots, x_d} \bigl\{ c(x_1^0, x_2, \dots, x_d) - \varphi_2(x_2) - \dots - \varphi_d(x_d) \bigr\}.$$
If $\tilde{\varphi}_1, \dots, \tilde{\varphi}_i$ are already defined, set 
\begin{equation*}
\begin{aligned}
\tilde{\varphi}_{i+1} (x_{i+1}^0) &=  \inf_{x_1, \dots, x_i, x_{i+2}, \dots, x_d} \bigl\{ c(x_1, \dots, x_i, x_{i+1}^0, x_{i+2}, \dots, x_d ) \\
& \hspace{30mm}  -  \tilde{\varphi}_1 (x_1) - \dots - \tilde{\varphi}_i (x_i) \\
& \hspace{30mm} - \varphi_{i+2}(x_{i+2}) - \dots - \varphi_d(x_d) \bigr\}.
\end{aligned}
\end{equation*}

The functions $\tilde{\varphi}_1, \dots, \tilde{\varphi}_d$ are measurable (in fact, upper semi-continuous), and constitute a $(G,c)$-splitting tuple.
\end{proof}

\begin{lemma}
If $G$ is $c$-cyclically monotone and finite, then it is $c$-splitting.
\end{lemma}
\begin{proof}
Immediate application of the definition of $c$-monotonicity and LP duality, cf. \cite{BeJu14}.
\end{proof}

\begin{lemma}
Let $c$ be continous,  $G$ be a $c$-splitting set, and $x^0 = (x_{1}^0, \dots, x_{d}^0) \in G$. Then there exists a measurable $(G, c)$-splitting tuple $(\varphi_{1}, \dots, \varphi_{d})$, such that 
\beo
\varphi_{i}(x_{i}) \leq c (x_{1}^0, \dots, x_{i-1}^0, x_{i}, x_{i+1}^0, \dots, x_{d}^0) ~\text{ for all } ~ x_{i} \in X_{i}, ~Ê i=1, \dots, d.
\eeo
\end{lemma}

\begin{proof}
By the assumptions  there is a measurable $(G,c)$-splitting-tuple $(\tilde{\varphi_{1}}, \dots, \tilde{\varphi_{d}})$. As $x^0 \in G$, we have 
\beo
\sum_{i=1}^d \tilde{\varphi}_{i} (x_{i}^0) = c(x^0).
\eeo
Hence, the values $\tilde{\varphi}_{i}(x_{i}^0)$ are all in $\RR$. 
Now define
\begin{eqnarray*}
\varphi_{1}: x_{1} \mapsto \tilde{\varphi}_{1} (x_{1}) + \tilde{\varphi}_{2}(x_{2}^0) + \dots + \tilde{\varphi}_{d}(x_{d}^0)Ê\\
\varphi_{i}: x_{i} \mapsto \tilde{\varphi}_{i} (x_{i}) - \tilde{\varphi}_{i}(x_{i}^0), ~\text{ for } i= 2, \dots, d.
\end{eqnarray*}
We have of $\sum_{i=1}^d \varphi_{i}(x_{i}) = \sum_{i=1}^d \tilde{\varphi}_{i}(x_{i})$, hence $(\varphi_{1}, \dots, \varphi_{d})$ is a $(G,c)$-splitting tuple with 
$\varphi_{1} (x_{1}^0) = c(x^0) \geq 0$ and $\varphi_{i}(x_{i}^0) = 0$ for $i=2, \dots, d$.
We hence have:
\begin{align*}
\varphi_{1} (x_{1})  &\leq c (x_{1}, x_{2}^0, \dots, x_{d}^0 ) ~ \text{ for all } x_{1} \in X_{1} \\
\varphi_{i} (x_{i}) &\leq c (x_{1}^0, \dots, x_{i-1}^0, x_{i}, x_{i+1}^0, \dots, x_{d}^0 ) - \varphi_{1}(x_{1}^0) \\
&\leq c (x_{1}^0, \dots, x_{i-1}^0, x_{i}, x_{i+1}^0, \dots, x_{d}^0 )
 ~ \text{ for all } x_{i} \in X_{i}. 
\end{align*}
\end{proof}

\begin{lemma}
Every $c$-cyclically monotone set $\Gamma$ is  $c$-splitting. 
\end{lemma}
\begin{proof}
(The result is trivial if $\Gamma$ is empty.) \\
We fix an element $x^0 \in \Gamma$. 
Define the functions $c_{i}: X_{i} \rightarrow [0, \infty)$ 
\beo
c_{i}: x_{i} \mapsto c(x_{1}^0, \dots, x_{i-1}^0,x_{i}, x_{i+1}^0, \dots, x_{d}^0).
\eeo
 For each finite subset $G$ of $\Gamma$, set $G' = G \cup \{ x^0 \}$. By the previous two lemmas, for each such $G'$ there is a $(G',c)$-splitting tuple with the components of the tuple bounded from above by $c_{1}, \dots, c_{d}$, respectively.
Now we define: 
\begin{align*}
\cG_{G} = \bigl\{ \varphi \equiv (\varphi_{1}, \dots, \varphi_{d}): \varphi \text{ is  a } (G', c)- \text{splitting tuple with } \\ \varphi_{i}(x_{i}) \leq c_{i}(x_{i}) \text{ for all } x_{i}\in X_{i}, i= 1, \dots, d \bigr\}.
\end{align*}
The sets  $\cG_{G'}$ are nonempty by our previous considerations. Note that they are closed in the topology of pointwise convergence on the \emph{compact} function space $\overline{\RR}^{X_{1}} \times \dots \times \overline{\RR}^{X_{d}}$. Also, the sets $\cG_{G'}$ have the finite intersection property: this is clear from 
\beo
\cG_{(G_{1} \cup G_{2})'} \subseteq \cG_{G_{1}'} \cap \cG_{G_{2}'}.
\eeo
Consequently, the set 
\beo
\cG = \bigcap_{G \subseteq \Gamma, \;  G \text{ finite}} \cG_{G'}
\eeo
is non-empty. It is easy to check that each of the tuples in $\cG$ is $(\Gamma, c)$-splitting. 
\end{proof}
\begin{proof}[Proof of Theorem 1.2.]
$\mu$ is concentrated on a $c$-cyclically monotone, and hence $c$-splitting set $\Gamma$. By the assumption on  $c$, for any  $x^0 = (x_{1}^0, \dots, x_{d}^0)$ in $\Gamma$ the functions 
\beo
c_{i}: x_{i} \mapsto c(x_{1}^0, \dots, x_{i-1}^0, x_{i}, x_{i+1}^0, \dots, x_{d}^0)
\eeo
are in $L_1(\mu_{i})$.
By Lemma 2.4, there is a measurable $(\Gamma, c)$-splitting tuple $(\varphi_{1}, \dots, \varphi_{d})$ such that 
\beo
\varphi_{i}(x_{i}) \leq c_{i}(x_{i}) ~~\text{ for all }x_{i} \in X_{i}, i=1, \dots, d.
\eeo
Hence, the functions $\varphi_{i}$ are all integrable against $\mu_{i}$, with the value of the integral in $[-\infty, \infty)$. 
Now take any $\mu' \in \cM(\mu_{1}, \dots, \mu_{d})$. 
We have, as $\mu$ is concentrated on the $c$-splitting set $\Gamma$, and $(\varphi_{1}, \dots, \varphi_{d})$ is $(\Gamma, c)$-splitting:
\beo
\int c \; d \mu = \sum \int \varphi_{i} \; d\mu_{i} \leq \int c \; d \mu'.
\eeo
\end{proof}

\bibliographystyle{alpha}
\bibliography{joint_biblio}

\end{document}